\documentclass[12pt]{amsart}
\usepackage{amsmath,amssymb, amsfonts,amscd,psfrag,graphicx,enumitem,tikz}
\usepackage{mathrsfs}
\usetikzlibrary{patterns}
\tikzstyle{vertex}=[circle, draw, inner sep=0pt, minimum size=6pt]
\newcommand{\lebn}

\theoremstyle{plain}
\newtheorem{prop}[equation]{Proposition}

\newtheorem{thm}[equation]{Theorem}
\newtheorem{fact}[equation]{Fact}
\newtheorem{cor}[equation]{Corollary}
\newtheorem{lem}[equation]{Lemma}

\theoremstyle{definition}

\newtheorem{defn}[equation]{Definition}
\newtheorem{rem}[equation]{Remark}

\numberwithin{equation}{section}

\newcommand{\Z}{\mathbb{Z}}

\newcommand{\E}{\mathcal{E}}

\newcommand{\CE}{\mathcal{C}}
\newcommand{\DE}{\mathcal{D}}
\newcommand{\FE}{\mathcal{F}}

\newcommand{\IE}{\operatorname{Ind}}

\newcommand{\ce}{\operatorname{c}}
\newcommand{\lk}{\operatorname{link}}
\newcommand{\dl}{\operatorname{del}}

\newcommand{\connH}{\operatorname{conn_h}}
\newcommand{\conn}{\operatorname{conn}}

\newcommand{\hk}{\operatorname{k}}

\begin{document}

\bibliographystyle{plain}

\title[Bounds on the connectivity of the independence complexes of hypergraphs]{Bounds on the connectivity of the independence complexes of hypergraphs}
\author{Demet Taylan}
 \pagestyle{plain} 
\address{Department of Mathematics, Yozgat Bozok University, Yozgat, Turkey.}
\email{demet.taylan@bozok.edu.tr}


\date{\today}


\begin{abstract}
We provide lower bounds on the connectivity of the independence complexes of hypergraphs. Additionally, we compute the homotopy types of the independence complexes of $d$-uniform properly-connected triangulated hypergraphs.
\end{abstract}

\maketitle

\section{Introduction}
By a hypergraph $\CE$ on a finite vertex set $V$ (or $V(\CE)$) we mean a family of pairwise incomparable subsets (edges) of $V$ whose every member has cardinality at least $2$. We identify the set of edges with the hypergraph $\CE$ itself. A $d$-uniform hypergraph is a hypergraph each of whose edges has cardinality $d$. For example, a simple graph is a $2$-uniform hypergraph. A vertex $v$ of $\CE$ is called an isolated vertex if there exists no edge of $\CE$ containing the vertex $v$. 

A simplicial complex $\Delta$ on a finite vertex set $V$ is a collection of subsets (faces) of $V$ that is closed under taking subsets such that $\{x\}\in \Delta$ for every $x\in V$. There are several ways to associate a simplicial complex to a hypergraph. For example, independence complex $\IE(\CE)$ of a hypergraph $\CE$ is the simplicial complex whose faces are the independent sets of vertices of $\CE$, i.e., the sets which do not contain any edge of $\CE$. 

The topology of independence complexes has been studied by several authors. A specific interest has been attracted towards the connectivity properties of $\IE(\CE)$ in the study of various combinatorial problems~\cite{aharoni, Bj2, Engstrom, Lovasz, Mesh, Mesh2}. In particular, homotopical connectivity proves to be an effective tool for dealing with independent systems of representatives (ISR) (\cite{aharoniberger}, Theorem 2.1) and matchable pairs (\cite{aharoni}, Theorem 4.5). It is therefore important to determine the homotopy types of simplicial complexes or to give bounds on the connectivity of simplicial complexes. The homotopical connectivity $\conn(\Delta)$ of a simplicial complex $\Delta$ is defined to be the greatest integer $k$ such that $\pi_i(\Delta)=0$ for every $-1\leq i\leq k$. In the case where $\pi_i(\Delta)=0$ for all $i$, we write $\conn(\Delta) = \infty$. In particular, $\conn(\emptyset): = -2$. Similarly, homological connectivity $\connH(\Delta)$ can be defined by replacing the homotopy groups $\pi_i(\Delta)$ by the reduced homology groups with integer coefficients $\tilde{H}_{i}(\Delta)$.

When $\Delta$ is the independence complex $\Delta = \IE(G)$ of a simple graph $G$, Aharoni, Berger and Ziv~\cite{aharoniberger} propose a recursively defined number $\psi(G)$ for each simple graph $G$ which gives a lower bound for the homotopical connectivity number $\conn(\IE(G))$ and conjecture that this bound is optimal. Even if Kawamura~\cite{Kawamura} verifies the conjecture for chordal graphs, the general claim was disproved in~\cite{Adamaszekbarmak,barmak}.

Our aim here is to generalize these results on the homotopical connectivity from independence complexes of graphs to independence complexes of hypergraphs. If $\CE$ is a hypergraph and $F$ is an edge, there are two key hypergraphs $\CE - F$  and $\CE:F$ obtained from $\CE$. The hypergraph $\CE - F$ denotes the hypergraph obtained by removing the edge $F$ from $\CE$. On the other hand, $\CE:F$ is the hypergraph on $V\setminus (F\cup N_{\CE}(F))$ whose edges are given by the minimal sets of $\{E\setminus F\colon E\;\textrm{is an edge of}\;\CE - F\}$ with cardinality at least $2$. Here, the vertex neighbour set $N_{\CE}(F) \subseteq V$ of an edge $F$ of $\CE$ is defined by $$N_{\CE}(F): = \bigcup_{E\in \CE\;\textrm{with}\; |E\setminus F| =1}^{} {E\setminus F}.$$

We now consider a recursively defined number $\psi (\CE)\in \Z_{\geq}\cup\{\infty\}$, where $\Z_{\geq}$ denotes the set of all non-negative integers, satisfying the following:

\begin{equation*}
\psi(\CE):=
\begin{cases}
0, & \text{if $V= \emptyset$},\\
\infty, & \text{if $V\neq \emptyset$ and $\CE=\emptyset$},\\
\max_{F\in\CE}\{\min\{\psi(\CE- F),\psi(\CE: F) + |F| -1 \}\}, & \text{otherwise}.\\
\end{cases}
\end{equation*}

We prove that if $\CE$ is a hypergraph, then the number $\psi(\CE)$ provides a lower bound for the homotopical connectivity $\conn(\IE(\CE))$ of the independence complex $\IE(\CE)$.

\begin{thm}\label{thm: connectivity bound for any given hypergraph} If $\CE$ is a hypergraph, then we have $\conn(\IE(\CE)) \geq \psi(\CE) - 2$.
\end{thm}

We next give a homotopical connectivity result on the independence complex of a hypergraph in terms of the order and the maximal degree. This generalizes the same result on graphs given by Engstr\"om~\cite{En} to hypergraphs. If $\CE$ is a hypergraph on $V$, the degree $d_\CE (v)$ of a vertex $v$ in $\CE$ is the number of edges containing the vertex $v$ and the maximal degree is denoted by $\Delta(\CE)$. For convention, we set $\Delta (\CE) = 1$ if $V=\emptyset$.

\begin{cor}\label{cor: homotopic connectivity bound depending vertex number} If $\CE$ is a hypergraph with $n$ vertices and maximal degree $\Delta(\CE)$, then $\IE(\CE)$ is $\left \lfloor{\frac{n-1}{2\Delta(\CE)}-1}\right \rfloor$-connected.
\end{cor}

The independence complexes have also attracted interest from algebraic points of view with respect to the edge ideals of hypergraphs. For instance, H\`a and van Tuyl~\cite{Tuyl,Tuyl2} explore connections between the algebraic invariants of edge ideals of hypergraphs encoded in their minimal free resolutions and the combinatorial properties of the underlying hypergraphs. In particular, H\`a and van Tuyl~\cite{Tuyl,Tuyl2} introduce $d$-uniform properly-connected triangulated hypergraphs for which the graded Betti numbers of their edge ideals are completely resolved recursively in terms of their subhypergraphs. In fact, $d$-uniform properly-connected triangulated hypergraphs generalize chordal graphs. When $G$ is a chordal graph, the bound $\psi(G) -2$ is optimal on the homotopical connectivity number of the independence complex $\IE(G)$. We generalize this result from chordal graphs to $d$-uniform properly-connected triangulated hypergraphs. More precisely, we prove that the bound $\psi(\CE) -2$ is optimal on the homotopical connectivity of the independence complex of a $d$-uniform properly-connected triangulated hypergraph $\CE$.

H\`a and van Tuyl~\cite{Tuyl,Tuyl2} further prove that the regularity of the edge ideal of a $d$-uniform properly-connected triangulated hypergraph is $(d-1)c_{\CE}+1$, where $\ce_{\CE}$ denotes the maximum number of pairwise $(d + 1)$-disjoint edges of the hypergraph. We prove that if $\CE$ is a $d$-uniform properly-connected triangulated hypergraph, then its independence complex $\IE(\CE)$ is either contractible or is homotopy equivalent to a wedge of spheres of dimension at most $(d - 1)\ce_{\CE} - 1$. Recall that a set of edges $\E\subseteq \CE$ is called pairwise $t$-disjoint if the distance $dist_{\CE}(F, G)$ between $F$ and $G$ in $\CE$, that is the length of the shortest path between $F$ and $G$ in $\CE$, satisfies $dist_{\CE}(F, G) \geq t$ for any two edges $F, G \in \E$. 

\begin{thm}\label{thm: triangulatd hypergraphs and homotopy type}
 Let $\CE$ be a $d$-uniform properly-connected triangulated hypergraph on $V$. Then, we have the following:

\begin{enumerate}[label=(\roman*)] 
\item There exists an edge $F$ of $\CE$ such that $\connH(\IE(\CE : F))\geq \connH(\IE(\CE)) - |F| + 1$ holds.
\item The independence complex $\IE(\CE)$ is either contractible or is homotopy equivalent to a wedge of spheres of dimension at most $(d - 1)\ce_{\CE} - 1$, where $\ce_{\CE}$ denotes the maximum number of pairwise $(d + 1)$-disjoint edges of $\CE$.
\end{enumerate}
\end{thm}

Our paper is organized as follows. In Section $2$, we review the background on hypergraphs and simplicial complexes, as well as necessary topological background. In Section $3$, we provide lower bounds on the connectivity number of the independence complexes of hypergraphs. We introduce properly-splitted hypergraphs and, in Section $4$, we compute the homotopy types of independence complexes of $d$-uniform properly-connected triangulated hypergraphs and show that properly-connected triangulated hypergraphs satisfy the property of being properly-splitted.

\section{Preliminaries}\label{pre}

\subsection{Hypergraphs}
A subhypergraph of a hypergraph $\CE$ on $V$ is the hypergraph in which each vertex and edge edge is contained in $V$ and $\CE$, respectively. For a given hypergraph $\CE$ on $V$ and a subset $A\subseteq V$, the~\emph{induced subhypergraph on} $A$ is the subhypergraph $\CE_A$ on $A$ with $\CE_A = \{E\in \CE\colon E\subseteq A\}$. The~\emph{$d$-complete hypergraph of order $n$} is the hypergraph in which the edges are the $d$-subsets of its vertex set $V$, where $|V| =n$. If $n< d$, a $d$-complete hypergraph of order $n$ is considered to be the hypergraph with $n$ isolated points.

If $\CE$ is a hypergraph on $V$, two distinct vertices $x,y\in V$ are said to be neighbours if there is an edge $F$ containing both the vertices $x$ and $y$. The neighbourhood of a vertex $x$ of a hypergraph $\CE$ on $V$ is the set $N(x) = \{y\in V\colon y\;\textrm{is a neighbour of} \; x\}$ (or $N_{\CE}(x)$). For a given hypergraph $\CE$ on $V$, a subfamily $\FE\subseteq \CE$ of a hypergraph $\CE$ is called~\emph{edgewise dominant} if a non-isolated vertex in $V$ is either contained in some edge of $\FE$ or has a neighbour contained in some edge of $\FE$ and $\epsilon (\CE)$ is given by $\epsilon (\CE) = min \{|\FE|\colon \FE \subseteq \CE\; \textrm{is edgewise dominant}\}$.

Given two edges $F$ and $G$ of a hypergraph $\CE$ with $|F| \geq |G|$, a~\emph{proper chain} is a sequence $(E_0 = F ,x_1,E_1,x_2,\dots, x_n, E_n = G)$, also denoted by $(E_0 = F ,E_1 ,\dots, E_n = G)$ if no confusion can arise, where $x_i$'s are distinct vertices and $E_j$'s are distinct edges of $\CE$ such that $x_1\in E_0$, $x_n\in E_n$, and $x_k,x_{k+1}\in E_k$ for $k = 1,\dots,n-1$ with $|E_i\cap E_{i+1}| = |E_{i+1}| -1$ for $i =0,\dots,n-1$. Moreover, the proper chain is called a~\emph{proper irredundant chain} if none of its subsequences is a proper chain from $F$ to $G$ and the~\emph{distance} $dist_{\CE}(F, G)$ between the edges $F$ and $G$ is defined as the minimum lenght of a proper irredundant chain connecting $F$ to $G$; i.e. $dist_{\CE}(F, G) = min \{l\colon (E_0 = F ,E_1 ,\dots, E_l = G) \;\textrm{is a proper irredundant chain}\}$. A $d$-uniform hypergraph is called a~\emph{properly-connected} hypergraph if for any two edges $F$ and $G$ whose intersection is non-empty, $dist_{\CE}(F, G) = d - |F\cap G|$ holds. Furthermore, if $\CE$ is a $d$-uniform properly-connected hypergraph, a set of edges $\E\subseteq \CE$ is called~\emph{pairwise $t$-disjoint} if $dist_{\CE}(F, G) \geq t$ for any two edges $F, G \in \E$.

Recall that an edge $F$ of a $d$-uniform properly-connected hypergraph $\CE$ on $V$ is called a~\emph{splitting edge} if there exists a vertex $z\in F$ such that $(F \setminus \{z\}) \cup \{x_i\} \in \CE$ for each $x_i\in N_{\CE}(F) = \{x_1,x_2,x_3,\dots, x_t\}$. H\`a and van Tuyl~\cite{Tuyl} provide the following relabeling lemma.

\begin{lem}\cite{Tuyl}\label{lem: d-uniform properly connected hypergraphs and relabeling} Let $\CE$ be a $d$-uniform properly-connected hypergraph. Let $F = E_0 = \{x_1,\dots, x_d\}$ and $G$ be any of its two edges satisfying $dist_{\CE}(F, G) = t \leq d$. Then, after relabeling, there exist edges $E_1, \dots, E_t$ such that $E_i = \{y_1, \dots, y_i, x_{i+1}, \dots, x_d\}$, $E_t = G$ and $y_i\notin E_j$ for any $j < i$.
\end{lem} H\`a and van Tuyl~\cite{Tuyl} prove that the property of being properly-connected is passed on to $\CE^{\geq}$:

\begin{lem}\cite{Tuyl}\label{lem: link is properly-connected hypergraph} For any edge $F$ of a $d$-uniform properly-connected hypergraph $\CE$ on $V$, the subhypergraph $\CE^{\geq} = \{G\in \CE\colon dist_{\CE}(F, G) \geq d+1\}$ is also a $d$-uniform properly-connected hypergraph.
\end{lem}

A $d$-uniform properly-connected hypergraph $\CE$ on $V$ is~\emph{triangulated} if the induced subhypergraph $\CE_A$  on any non-empty subset $A\subseteq V$ contains a vertex $v\in A$ such that the set $N_{\CE_A}(v)$ induces a $d$-complete hypergraph of order $|N_{\CE_A}(v)|$ and $v$ appears at most twice in any proper irredundant connected chain in $\CE_A$ (\cite{Tuyl2}).

\begin{thm}\label{thm: triangulated hypergraphs and a special splitting edge}\cite{Tuyl2} Let $\CE$ be a triangulated hypergraph on $V$ and let $v\in V$ be a vertex of $\CE$ such that the induced subhypergraph $\CE_{N(v)}$ is a $d$-complete hypergraph, and $v$ appears at most twice in any irredundant chain in $\CE$. Then, any edge $E$ of $\CE$ that contains the vertex $v$ is a splitting edge and the subhypergraphs $\CE - E$ and $\CE^{\geq} = \{G\in \CE\colon dist_{\CE}(E, G) \geq d+1\}$ are triangulated hypergraphs.
\end{thm}

\subsection{Simplicial Complexes}\label{preliminaries: simplicial complex}
If $\Delta$ is a simplicial complex on $V$ and $U\subset V$, then the complex $\Delta_U:=\{\sigma\colon \sigma\in \Delta\;\textrm{,}\;\sigma\subseteq U\}$ is called the \emph{induced subcomplex} by $U$. There are two particularly important subcomplexes associated with a simplicial complex. To be more precise, if $\sigma\in \Delta$ is a face of a simplicial complex $\Delta$, then the \emph{link} $\lk_\Delta(\sigma)$ and the \emph{deletion} $\dl_\Delta(\sigma)$ are defined respectively by $\lk_\Delta(\sigma)=\{\tau\in\Delta\colon\tau\cap \sigma=\emptyset\;\textrm{and}\;\tau\cup\sigma\in\Delta\}$ and $\dl_\Delta(\sigma)=\{\tau\in\Delta\colon\tau\nsupseteq\sigma\}$. If $x$ is a vertex of $\Delta$, we abbreviate $\dl_{\Delta}(\{x\})$ and $\lk_\Delta(\{x\})$ to $\dl_{\Delta}(x)$ and $\lk_{\Delta}(x)$, respectively. 

Let $\Delta$ be a simplicial complex on $V$. Then, we can associate to $\Delta$ the hypergraph $\CE(\Delta)$ of minimal non-faces of $\Delta$. It is easy to verify that $\IE(\CE(\Delta)) = \Delta$.

A simplicial complex $\Delta$ is called $\Z$-acyclic if $\tilde{H}_{i}(\Delta) = 0$ for all $i$. Recall that a simplicial complex $\Delta$ is contractible if and only if $\Delta$ is simply-connected and $\Z$-acyclic (\cite{Bj}). Recall also that if $\Delta$ is a simplicial complex, then $\connH(\Delta) \geq \conn(\Delta)$ holds. In particular, if $\Delta$ is a simply-connected simplicial complex, then the homological connectivity equals to homotopical connectivity by the Hurewicz theorem (\cite{AH}).

Throughout this paper, $\mathbb{S}^n$ will denote the $n$-dimensional sphere. If two topological spaces $X$ and $Y$ are homotopy equivalent, we denote it by $X \simeq Y$. The join and the wedge of two topological spaces $X$ and $Y$ are denoted by $X*Y$ and $X\vee Y$, respectively. The suspension of a topological space $X$ will be denoted by $\Sigma(X)$.

The following fact is well-known.

\begin{thm}\label{Aharoni, Berger, Ziv's proposed function gives lower bound for connectivity of independence complex of a simple graph}\cite{aharoniberger} Let $\psi$ be a function from the class of all graphs $G=(V,E)$ to the set $ \Z_{\geq}\cup\{\infty\}$ defined by

\begin{equation*}
\psi(G)=
\begin{cases}
0, & \text{if $V= \emptyset$},\\
\infty, & \text{if $V\neq \emptyset$ and $E =\emptyset$},\\
\max_{e\in E}\{\min\{\psi(G - e),\psi(G_{V \setminus (e \cup N(e))}) + 1\}\}, & \text{otherwise}.\\
\end{cases}
\end{equation*}

Then, $\conn(\IE(G))\geq \psi(G) -2$ holds. 
\end{thm}

An explicit proof of Theorem~\ref{Aharoni, Berger, Ziv's proposed function gives lower bound for connectivity of independence complex of a simple graph} were given by~\cite{Adamaszekbarmak} and one application can be found in~\cite{aharonihoward}. Aharoni, Berger and Ziv~\cite{aharoniberger} conjectured that $\conn(\IE(G)) =\psi(G) - 2$ holds. This was confirmed for chordal graphs by Kawamura~\cite{Kawamura}. However, Barmak~\cite{barmak} disproved the claim by using a result of recursion theory. Moreover, Adamaszek and Barmak~\cite{Adamaszekbarmak} provided some examples for which the inequality is strict.

We require the following known facts from algebraic topology~\cite{Bj, Milnor, Munkres}.

\begin{thm}\label{union of complexes and homotopy}\cite{Bj} Let $\Delta_0,\dots, \Delta_n$ be contractible simplicial complexes and assume that $\Delta_i \cap \Delta_j \subseteq \Delta_0$ for all $1 \leq i < j \leq n$. Then, $\bigcup_{i=0}^n \Delta_i \simeq \vee_{i=1}^n \Sigma(\Delta_0 \cap \Delta_i)$.
\end{thm}

\begin{thm}\cite{Milnor}\label{thm: a condition for a join to be simply-connected space}
If $X$ is a path-connected space and $Y$ is any non-empty space, then the join $Y * X$ is simply-connected.
\end{thm}

\begin{thm}\cite{Munkres}\label{thm: join and homology}
Assume $X * Y$ exists. If $Y\simeq \mathbb{S}^{n-1}$, then for all $i$, $\tilde{H}_{i+n}(X * Y)$ is isomorphic to $\tilde{H}_i(X)$.
\end{thm}

\section{The connectivity of independence complexes of hypergraphs}\label{section: generalization of Aharoni, Berger and Ziv Teorem}

In this section, we prove that $\psi(\CE)$ provides a lower bound for the connectivity of the independence complex $\IE(\CE)$ of a hypergraph $\CE$. In the case of the independence complex of a simple graph, this reduces to a result (mentioned in Section~\ref{preliminaries: simplicial complex} as Theorem~\ref{Aharoni, Berger, Ziv's proposed function gives lower bound for connectivity of independence complex of a simple graph}) of Aharoni, Berger and Ziv in~\cite{aharoniberger}. We then prove that $\psi(\CE) = \conn(\IE(\CE))$ in the case when $\IE(\CE)$ is not simply-connected. Moreover, we show that $\psi(\CE)$ and $\conn(\IE(\CE))$ can take arbitrary values $l,k$ with $3\leq l\leq k$. Furthermore, we provide a lower bound on the homotopical connectivity number of the independence complex of a hypergraph in terms of the order and the maximal degree. Our departure in this direction begins by generalizing the results of~\cite{Adamaszekbarmak}.

For a given edge $F$ of a hypergraph $\CE$ on $V$, it is easy to verify that $$\IE(\CE:F) = \lk_{\IE(\CE-F)}(F)$$ and $$\IE(\CE-F) = \IE(\CE) \cup (||F|| * \IE(\CE:F)),$$ where $$||F|| = \IE(\CE)_F\cup\{F\}.$$

The following is immediate by the Mayer-Vietoris long exact sequence and generalizes Claim 3.1 of Meshulam in~\cite{Mesh}.

\begin{thm}\label{thm: exact sequence of homology groups-flagization} Let $\CE$ be a hypergraph and let $F$ be an edge of $\CE$. Then, we have the following long exact sequence:
\begin{equation*}
\dots\rightarrow \tilde{H}_{i-|F|+1}(\IE(\CE:F))\xrightarrow{} \tilde{H}_{i}(\IE(\CE))\xrightarrow{} \tilde{H}_{i}({\IE(\CE - F)})\rightarrow \tilde{H}_{i-|F|}(\IE(\CE:F))\rightarrow \dots 
\end{equation*}
\end{thm}

\begin{proof} Suppose that $\CE$ is a hypergraph and $F$ being an edge of $\CE$. Then, $\IE(\CE - F) = \IE(\CE) \cup (||F||* \IE(\CE : F))$ and $\IE(\CE)\cap (||F||* \IE(\CE : F)) = \IE(\CE)_F *\IE(\CE : F)$ hold, where $||F|| = \IE(\CE)_F \cup\{F\}$. Note that $\tilde{H}_i(\IE(\CE)\cap (||F||* \IE(\CE : F)))$ is isomorphic to $\tilde{H}_{i-|F|+1}(\IE(\CE : F))$ by Teorem~\ref{thm: join and homology}. If we apply the Mayer-Vietoris long exact sequence, we obtain the desired long exact sequence.
\end{proof}

\begin{lem}\label{lem: some relations for the number of minimal non-faces}

Let $\CE$ be a hypergraph on $V$ and let $F$ be an edge of $\CE$. Then, we have the following:
\begin{enumerate}[label=(\roman*)]
\item $|\CE|=|\CE - F| + 1$,
\item $|\CE|\geq |\CE:F| + 1$,
\end{enumerate} 
\end{lem}

\begin{proof}
Claim (i) is obvious. To prove (ii), let $\CE$ be a hypergraph on $V$ and let $F$ be an edge of $\CE$. Assume that $S$ is an edge of $\CE: F$ satisfying $S\in \IE(\CE)$. Then, there must exist an edge $T$ of $\CE$ such that $T = S \cup A$ for some subset $A\subset F$, since otherwise $F\cup S \in \IE(\CE - F)$, yielding $S\in \IE(\CE : F)$. We can now define a one-to-one function $f$ from $\CE: F$ to $\CE -F$ by selecting one such $T$ and define $f(S)=T$. This completes the proof.
\end{proof}

We have the following homological connectivity result for independence complexes of hypergraphs.

\begin{cor}\label{cor: homological connectivity bound given by a function} For any hypergraph $\CE$, we have $\connH(\IE(\CE))\geq \psi(\CE) -2$.
\end{cor}

\begin{proof}
We proceed by induction on $|\CE|$. For the base case $|\CE| = 0$, the claim is trivial. Suppose now that $|\CE| \geq 1$ and choose any edge $F$ of $\CE$ satisfying $\psi(\CE) = \min\{\psi(\CE - F),\psi(\CE: F) + |F| -1\}$. The induction hypothesis then provides that $\psi(\CE - F) \leq \connH(\CE - F) + 2$ and $\psi(\CE: F) \leq \connH(\CE: F) +2$. Thus, we have $\tilde{H}_{i}(\IE(\CE - F)) = 0$ for every $i \leq \psi(\CE -F) - 2$ and $\tilde{H}_{i}(\IE(\CE : F)) = 0$ for every $i \leq \psi(\CE : F)-2$. It therefore follows that $\tilde{H}_{i}(\IE(\CE - F) = 0$ for all $i \leq \psi(\CE) - 2$ and $\tilde{H}_{i}(\IE(\CE : F)) = 0$ for all $i \leq \psi(\CE) - |F| - 1$, since $\psi(\CE - F) \geq \psi(\CE)$ and $\psi(\CE : F)\geq \psi(\CE) - |F| + 1$ by the assumption on $F$. Hence applying Theorem~\ref{thm: exact sequence of homology groups-flagization}, we obtain $\tilde{H}_{i}(\IE(\CE)) = 0$ for every $i \leq \psi(\CE) - 2$. This completes the proof.

\end{proof}

We next provide a sufficient condition which guaranties that $\IE(\CE)$ is simply-connected.
 
\begin{thm}\label{thm: path-simply connected complexes} Let $\CE$ be a hypergraph.

\begin{enumerate}[label=(\roman*)]
\item If $\psi(\CE) \geq 2$, then $\IE(\CE)$ is path-connected.
\item If $\psi(\CE) \geq 3$, then $\IE(\CE)$ is simply-connected.
\end{enumerate}
\end{thm}

\begin{proof}
Assertion (i) is true, since $\connH(\IE(\CE))\geq 0$ holds by Corollary~\ref{cor: homological connectivity bound given by a function} for any hypergraph satisfying $\psi(\CE) \geq 2$.

To verify (ii), we proceed by induction on $|\CE|$. Suppose that $|\CE| = 0$. Then, $\IE(\CE)$ is a simplex, since $\psi(\CE) \geq 3$. This establishes the base case of the induction. Assume now that $|\CE| \geq 1$. We can now choose an edge $F$ satisfying $$\psi(\CE - F) \geq 3$$ and $$\psi(\CE : F) + |F| - 1 \geq 3,$$ since $\psi(\CE) \geq 3$. By induction, $\IE(\CE - F)$ is simply-connected. Note that $$\IE(\CE - F) = \IE(\CE) \cup (||F||* \IE(\CE : F)),$$ where $||F|| = \IE(\CE)_F \cup \{F\}$ and the intersection $$\IE(\CE) \cap (||F||* \IE(\CE : F)) = \IE(\CE)_F * \IE(\CE : F)$$ is simply-connected by Theorem~\ref{thm: a condition for a join to be simply-connected space}, since for the case $|F| = 2$ we have $2\leq \psi(\CE : F) \leq \connH(\IE(\CE : F)) + 2$ and whence $\IE(\CE : F)$ is path-connected, while for the case $|F| \geq 3$, the induced subcomplex $\IE(\CE)_F$ is path-connected and note that if $|F|=3$, then $\IE(\CE : F)$ is non-empty, since $\psi(\CE : F) \geq 3-|F| +1 = 1.$ The simplicial complex $\IE(\CE)$ is path-connected by Theorem~\ref{thm: path-simply connected complexes} (i), since $\psi(\CE) \geq 3$. It then follows by Van Kampen's theorem that $\pi_1(\IE(\CE - F))$ is the free product of $\pi_1(\IE(\CE))$ and $\pi_1(||F|| * \IE(\CE : F))$. This implies that $\pi_1(\IE(\CE)) = 0$.
\end{proof}

We are now ready to prove Theorem~\ref{thm: connectivity bound for any given hypergraph}.

\begin{proof}[{\bf Proof of Theorem~\ref{thm: connectivity bound for any given hypergraph}}] 
For the case $\psi(\CE) \geq 3$, the claim obviously holds by Theorem~\ref{thm: path-simply connected complexes} (ii) and Hurewicz theorem together with Corollary~\ref{cor: homological connectivity bound given by a function}. Now, suppose that $\psi(\CE) = 0$. Then, $\IE(\CE) =\{\emptyset\}$, hence $\conn(\IE(\CE)) = -2$, implying $\conn(\IE(\CE)) = \psi(\CE) -2$. Assume next that $\psi(\CE) = 1$. We conclude in this case that $\IE(\CE) \neq \{\emptyset\}$, hence $\conn(\IE(\CE)) \geq -1$. Thus, it follows that $\conn(\IE(\CE)) + 2 \geq 1 = \psi(\CE)$. Finally, suppose that $\psi(\CE) = 2$. Then, $\IE(\CE)$ is path-connected by Theorem~\ref{thm: path-simply connected complexes} (i). Hence, $\conn(\IE(\CE)) + 2 \geq 2 = \psi(\CE)$. This completes the proof.

\end{proof}

\begin{prop}\label{prop: value of the join of simplicial complexes} The equality $\psi(\CE(\IE(\CE_1) * \IE(\CE_2))) = \psi(\CE_1) + \psi(\CE_2)$ holds for any given two hypergraphs $\CE_1$ and $\CE_2$.
\end{prop}

\begin{proof} We prove the statement by induction on $|\CE_1\cup \CE_2|$. Clearly, the base case is true. Assume that $|\CE_1\cup \CE_2|\geq 1$. Then, it follows by induction that

\begin{align*}&\psi(\CE(\IE(\CE_1) * \IE(\CE_2)))\\
&\qquad= \max_{F\in\CE_1}\{\min\{\psi(\CE(\IE(\CE_1) * \IE(\CE_2)) - F), \psi(\CE(\IE(\CE_1) * \IE(\CE_2)) : F) + |F| -1\}\}\\
&\qquad = \max_{F\in\CE_1}\{\min\{\psi(\CE(\IE(\CE_1 - F) * \IE(C_2))),\psi(\CE(\IE(\CE_1 : F) * \IE(\CE_2))) + |F| -1\}\}\\
&\qquad = \max_{F\in\CE_1}\{\min\{\psi(\CE_1 - F),\psi(\CE_1 : F) + |F|-1\}\} + \psi(\CE_2)\\
&\qquad =\psi(\CE_1) + \psi(\CE_2)
\end{align*}
or

\begin{align*}&\psi(\CE(\IE(\CE_1) * \IE(\CE_2)))\\
& \qquad = \max_{F\in\CE_2}\{\min\{\psi(\CE(\IE(\CE_1) * \IE(\CE_2)) - F),\psi(\CE(\IE(\CE_1) * \IE(\CE_2)): F) + |F| -1\}\}\\
& \qquad =\max_{F\in\CE_2}\{\min\{\psi(\CE(\IE(\CE_1) * \IE(\CE_2 - F))),\psi(\CE(\IE(\CE_1) * \IE(\CE_2 : F))) + |F| -1\}\}\\
& \qquad =\max_{F\in\CE_2}\{\min\{\psi(\CE_2 - F),\psi(\CE_2 : F) + |F| -1\}\} + \psi(\CE_1)\\
& \qquad  =\psi(\CE_1) + \psi(\CE_2).
\end{align*} This completes the proof.
\end{proof}

\begin{lem}\label{lem: a condition impliying zero sphere to be contractible in independence complex}
If $\psi(\CE)=1$ holds for a hypergraph $\CE$, then $\mathbb{S}^0$ is not contractible in $\IE(\CE)$.
\end{lem}

\begin{proof}
We prove this by induction on $|\CE|$. Suppose that $\psi(\CE)=1$ holds for a hypergraph $\CE$. Then, we must have either $\psi(\CE- F) = 1$ or $\psi(\CE: F) + |F| -1 = 1$ for any edge $F$ of $\CE$. If $\psi(\CE - F) = 1$ holds for some $F\in \CE$, then, by induction, $\mathbb{S}^0$ is not contractible in $\IE(\CE - F)$. This gives that $\mathbb{S}^0$ is not contractible in $\IE(\CE)$. It is thus enough to examine the case when for every edge $F$ of $\CE$ the hypergraph $\CE: F$ satisfies $\psi(\CE: F) = 2 - |F|$. Note then that, in this case, we must have $|F| = 2$ and $\IE(\CE: F )= \{\emptyset\}$ for any edge $F$ of $\CE$. This gives that $\IE(\CE)$ is a disjoint union of simplices. Since $\IE(\CE)$ is not a simplex, $\IE(\CE)$ is not connected. This yields that $\mathbb{S}^0$ is not contractible in $\IE(\CE)$. 
\end{proof}

\begin{cor}\label{cor: not simply connected implies tight bound} If the independence complex $\IE(\CE)$ of a hypergraph $\CE$ is not simply-connected, then $\psi(\CE) = \conn(\IE(\CE)) + 2$ holds.
\end{cor}

\begin{proof} If $\IE(\CE)$ is not simply-connected, then $\psi(\CE) \leq \conn(\IE(\CE)) + 2 \leq 2$ holds by Theorem~\ref{thm: connectivity bound for any given hypergraph} and the result follows immediately from Theorem~\ref{thm: path-simply connected complexes} and Lemma~\ref{lem: a condition impliying zero sphere to be contractible in independence complex} and the fact that $\psi(\CE) = 0$ if and only if $\IE(\CE) = \{\emptyset\}$.
\end{proof}

Adamaszek and Barmak prove that for any $l, k \in \Z_{\geq}\cup\{\infty\}$ with $k\geq l\geq 3$, there exists a simple graph $G$ such that $\psi(G) = l$ and $\conn(\IE(G)) = k - 2$ (\cite{Adamaszekbarmak}, Proposition 5). This shows that there are counterexamples to Aharoni-Berger-Ziv conjecture. We generalize Proposition 5 in~\cite{Adamaszekbarmak} from graphs to hypergraphs.

\begin{thm}\label{thm: existenxe of a hypergraph with some certain properties}
For any $l, k \in \Z_{\geq}\cup\{\infty\}$ with $k\geq l\geq 3$, there exists a hypergraph $\CE$ satisfying $\psi(\CE) = l$ and $\conn(\IE(\CE)) = k - 2$.
\end{thm}

\begin{proof}
Any hypergraph with an isolated vertex satisfies the statement for $l = \infty$. For finite values of $l$, it is enough to prove the statement for $l = 3$, since for a hypergraph $\CE$ with a finite value $l = \psi(\CE)$ and $\conn(\IE(\CE)) = k -2 \geq 1$, we have $\psi(\CE(\IE(\CE) * \IE(\CE'))) =\psi(\CE) + 1$ by Proposition~\ref{prop: value of the join of simplicial complexes} and $\conn(\IE(\CE \cup \CE')) = \conn(\IE(\CE) * \IE(\CE')) = \connH(\IE(\CE) * \IE(\CE')) = \connH(\IE(\CE)) + 1 = \conn(\IE(\CE)) + 1$, where $\CE' = \{\{a,b\}\}$ is a hypergraph on $\{a,b\}$ with $a, b\notin V(\CE)$.

We next consider the $\Z$-acyclic noncontractible complex on $10$ vertices whose facets (inclusion-wise maximal faces) is given by

\begin{align*}
   & \{1,2,4\} \hspace{5 mm} \{1,2,5\} \hspace{5 mm}\{1,3,6\}\hspace{5 mm} \{1,3,8\}\hspace{5 mm} \{1,3,10\}\hspace{5 mm} \{1,4,8\} \hspace{5 mm}\{1,4,9\}\hspace{5 mm} \{1,5,7\}\\
  & \{1,5,10\}\hspace{3 mm} \{1,6,7\}\hspace{5 mm} \{1,6,9\}\hspace{5 mm}\{2,3,5\}\hspace{5 mm} \{2,3,7\}\hspace{7 mm} \{2,3,8\}\hspace{5 mm} \{2,4,6\} \hspace{5 mm}\{2,4,10\}\\
  & \{2,6,7\} \hspace{5 mm}\{2,6,8\}\hspace{5 mm} \{2,8,10\}\hspace{3 mm} \{3,5,6\}\hspace{5 mm} \{3,5,9\}\hspace{7 mm} \{3,7,9\}\hspace{5 mm} \{3,7,10\}\hspace{3 mm} \{4,5,6\}\\
  & \{4,5,7\} \hspace{5 mm}\{4,5,8\}\hspace{5 mm} \{4,7,9\}\hspace{5 mm} \{4,7,10\} \hspace{3 mm}\{5,8,9\} \hspace{7 mm}\{5,8,10\} \hspace{3 mm} \{6,8,9\}
\end{align*} (\cite{BLutz}).
Consider now the hypergraph $\CE(\Delta)$ of minimal non-faces of the above complex $\Delta$. Let $\{F\}$ be a hypergraph on a set $F$ with $F \cap V(\CE(\Delta)) = \{z\}$ and $|F| = k \geq 3$ and consider the hypergraph $\CE^F = \CE(\Delta) \cup \{F\} \cup \CE^z$ on $V(\CE(\Delta)) \cup (F\setminus \{z\})$, where $\CE^z$ is the hypergraph given by $\CE^z = \{\{x,y\} \colon x\in V(\CE(\Delta)) \setminus \{z\}, y\in F\setminus \{z\}\}$. We have $\IE(\CE^F) = \IE(\CE(\Delta)) \vee \IE(\{F\})$. We then obtain that $\psi(\CE^F) = \conn(\IE(\CE^F)) + 2 = 2$ by Corollary~\ref{cor: not simply connected implies tight bound}, since $\IE(\CE^F)$ is not simply-connected. Note that we have $$\psi(\CE(\IE(\CE^F) * \IE(\{\{a,b\}\}))) = 2 + 1 =3,$$ by Proposition~\ref{prop: value of the join of simplicial complexes} and $$\conn(\IE(\CE^F) * \IE(\{\{a,b\}\})) = k - 2$$ for  a hypergraph $\{\{a,b\}\}$ on $\{a,b\}$ with $V(\CE^F) \cap \{a,b\} = \emptyset$.

For the case $l = 3$ and $k =\infty$, take a hypergraph $\{\{c,d\}\}$ on $\{c,d\}$ satisfying $V(\CE(\Delta)) \cap \{c,d\} = \emptyset$ and consider the complex $\IE(\CE(\Delta)) * \IE(\{\{c,d\}\})$. Note that for the hypergraph $\CE(\IE(\CE(\Delta)) * \IE(\{\{c,d\}\}))$, we have $$\psi(\CE(\IE(\CE(\Delta)) * \IE(\{\{c,d\}\}))) = 3.$$ Note also that the complex $\IE(\CE(\Delta)) * \IE(\{\{c,d\}\})$ is simply-connected by Theorem~\ref{thm: a condition for a join to be simply-connected space} and thus the complex $\IE(\CE(\Delta)) * \IE(\{\{c,d\}\})$ is contractible. This completes the proof.
\end{proof}

\begin{defn} Let $f$ be a function from the family of hypergraphs to the set $\Z_{\geq}\cup\{\infty\}$ and let $\CE$ be a hypergraph. Then, $f$ is called~\emph{inductive} on $\CE$ if the followings hold:

\begin{enumerate}[label = (\roman*)]
\item  $f(\CE) = \infty$ for $\CE = \emptyset$ and $V\neq \emptyset$.
\item $f(\CE) = 0$ whenever $V = \emptyset$.
\item If $\CE\neq \emptyset$, then $f(\CE) \leq f(\CE - K)$ and $f(\CE) \leq f(\CE : K) + |K| - 1$ hold for some edge $K$ of $\CE$ such that $f$ is inductive on $\CE - K$ and $\CE : K$.
\end{enumerate}
\end{defn}
 
For instance, $\conn (\IE(G)) + 2$ gives rise to an inductive function on the class of chordal graphs, and in particular the inequality $\conn (\IE(G)) + 2 \leq \psi (G)$ holds for every chordal graph $G$ (\cite{A}, Corollary 5.6 b)). We next generalize this latter inequality to hypergraphs as follows.

\begin{lem}\label{lem: hereditary value of a hypergraph} Let $f$ be a function from the family of hypergraphs to the set $\Z_{\geq}\cup\{\infty\}$ and let $\CE$ be a hypergraph. If $f$ is inductive on $\CE$, then $\psi (\CE) \geq f(\CE)$ holds.
\end{lem}

\begin{proof} We prove the statement by induction on the number of edges of $\CE$. Suppose that $f$ is inductive on a hypergraph $\CE$. When $\CE=\emptyset$, the claim is obviously true. Assume now that $|\CE|\geq 1$. Then, there exists an edge $K$ of $\CE$ such that $f(\CE) \leq f(\CE - K)$ and $f(\CE) \leq f(\CE : K) + |K| - 1$ hold and $f$ is inductive on the hypergraphs $\CE - K$ and $\CE : K$. Then,
\begin{align*}
f(\CE) &\leq \min\{f(\CE - K), f(\CE : K) + |K| - 1\} \\& \leq \min\{\psi(\CE - K), \psi(\CE : K) + |K| - 1\} \\& \leq \max_{F\in\CE}\{\min\{\psi(\CE - F), \psi(\CE : F) + |F| - 1\}\} \\& = \psi(\CE).
\end{align*}
This completes the proof.
\end{proof}

There are several domination related parameters for graphs. One such example is the total domination number of a graph. Recall that a subset $S\subseteq V$ of the vertex set $V$ of a graph $G$ is called a~\emph{total dominating set} of $G$ if $\bigcup_{s\in S} N(s) =V$. Moreover, the minimum cardinality of a total dominating set of $G$ is called the~\emph{total domination number} of $G$ and is denoted by $\gamma_t(G)$. Furthermore, the notion of total domination were extended to simplicial complexes (\cite{aharoni}). More precisely, Aharoni and Berger~\cite{aharoni} introduce a domination parameter $\tilde{\gamma}(\Delta)$ of a simplicial complex $\Delta$ on $V$, which is defined as the minimal size of a set $A\subseteq V$ such that $\tilde{sp}_{\Delta}(A) = V,$ where $$\tilde{sp}_{\Delta}(A) =\{v\in V\colon \textrm{there exists some face}\; \sigma\subseteq A\;\textrm{such that}\; \sigma\cup \{v\} \notin \Delta\}.$$ Clearly, when $G$ is a simple graph, then the equality $\gamma_t(G) = \tilde\gamma(\IE(G))$ holds. We reprove the following result (i) of~\cite{aharoni} whose homological version was given in~\cite{Mesh} for the graph case and prove (ii) whose homological version was given in~\cite{DaO}. 

\begin{thm}\label{relationships between invariants and connectivity numbers}
Let $\CE$ be a hypergraph on $V$.

\begin{enumerate}[label=(\roman*)]

\item\cite{aharoni} If $\hk(\CE)\in \Z_{\geq}\cup\{\infty\}$ is the number given by $\hk(\CE)= \left \lceil{\frac{\tilde{\gamma}(\IE(\CE))}{2}}\right \rceil$, then we have $\conn(\IE(\CE)) \geq \hk(\CE) -2$.

\item For the hypergraph $\CE$, we have $\conn(\IE(\CE)) \geq \epsilon(\CE) -2$. 

\end{enumerate}
\end{thm}

\begin{proof}
  To prove (i), let $F$ be any edge of a hypergraph $\CE$. Note then that $$\hk(\CE) = \left \lceil{\frac{\tilde{\gamma}(\IE(\CE))}{2}}\right\rceil \leq \left \lceil{\frac{\tilde{\gamma}(\IE(\CE - F))}{2}}\right\rceil =\hk(\CE - F)$$ and $$\hk(\CE) = \left \lceil{\frac{\tilde{\gamma}(\IE(\CE))}{2}}\right\rceil \leq \left \lceil{\frac{\tilde{\gamma}(\IE(\CE : F)) + |F|}{2}}\right\rceil \leq \hk(\CE : F) + |F| -1.$$

Taking Lemma~\ref{lem: hereditary value of a hypergraph} and Theorem~\ref{thm: connectivity bound for any given hypergraph} into account, we conclude that $\conn(\IE(\CE)) \geq \hk(\CE)-2$.

To verify (ii), let $F$ be an edge of $\CE$. It is obvious that $\epsilon (\CE - F) \geq \epsilon (\CE)$. Assume now that $\FE_{\CE : F} \subseteq \CE : F$ is an edgewise dominant set of $\CE : F$ with $|\FE_{\CE : F}| = \epsilon (\CE : F)$. If $\FE_{\CE : F}\subseteq \CE$, then $\FE_{\CE : F} \cup \{F\}$ is an edgewise dominant set of $\CE$. Suppose now that there exists some edge $E\in \FE_{\CE : F}$ of $\CE : F$ with $E\notin \CE$. This implies that there is some subset $M\subsetneq F$ such that $M\cup E\in \CE$. We now remove $E$ from $\FE_{\CE : F}$ and instead add $E' : = M\cup E$ to the set $\FE_{\CE : F}$. Let $\FE_{\CE : F}'$ be the family obtained from $\FE_{\CE : F}$ by repeating the same procedure for all the edges $E\in \FE_{\CE : F}$ satisfying the property that $E\in \CE : F$ but $E\notin \CE$, until no such edge exists. Note then that if $|N_{\CE}(F)|=0$, then $\FE_{\CE : F}'$ is an edgewise dominant set for $\CE$ with cardinality $\epsilon (\CE : F)$ and if $|N_{\CE}(F)|\geq 1$, then $\FE_{\CE : F}'\cup\{F\}$ is an edgewise dominant set for $\CE$ with cardinality $\epsilon (\CE : F) +1$. This completes the proof.
\end{proof}

A well-known fact for the total domination number of a graph $G$ is that $\gamma_t(G) = \tilde{\gamma}(\IE(G))\geq \frac{n}{\Delta(G)}$ holds (\cite{aharoni}). We next present a lower bound for $\tilde{\gamma}(\IE(\CE))$, where $\CE$ is a hypergraph.

\begin{lem} \label{lem: bound for domination number} Let $\CE$ be a hypergraph with $n$ vertices. Then, $\tilde{\gamma}(\IE(\CE))\geq \frac{n}{\Delta(\CE)}$ holds.

\end{lem}

\begin{proof} Suppose $\CE$ is a hypergraph on $V=\{v_1,v_2,\dots,v_n\}$. Let $A\subseteq V$ be a set with minimal size satisfying $\tilde{sp}_{\IE(\CE)}(A)=V$. For each edge $G$ of $\CE$ and for each $a\in A$, let $\sigma^G_{a}\subseteq G$ be the $(|G|-2)$-face of the induced subcomplex $\IE(\CE)_G$ such that $\sigma^G_a\subseteq A$ and $a=v_i\in \sigma^G_{a}$, where $v_i$ is the first vertex in the list of vertices of $\sigma^G_{a}$. Here, the vertices of $\sigma^G_{a}$ is listed in the order induced by the cyclic order of the vertices of $G\in \CE$. Note that

  $$|\tilde{sp}_{\IE(\CE)}(A)|=|\bigcup_{a\in A}\bigcup_{G\in \CE} \{v\in V\colon \sigma^G_a\cup \{v\} =G\}|.$$

  It then follows that

  \begin{align*}  n = |\tilde{sp}_{\IE(\CE)}(A)| &= |\bigcup_{a\in A}\bigcup_{G\in \CE}\{v\in V\colon \sigma^G_a\cup \{v\} =G\}|\\
    & \leq \sum_{a\in A}\sum_{G\in \CE} |\{v\in V\colon \sigma^G_a\cup \{v\} =G\}|\\  & \leq |A|d_\CE (a) \\
    & \leq |A|\Delta(\CE) \\
  &=\tilde{\gamma}(\IE(\CE)) \Delta(\CE).
\end{align*}

Thus, we have $\tilde{\gamma}(\IE(\CE))\geq \frac{n}{\Delta(\CE)}$. 
\end{proof}

Combining Theorem~\ref{relationships between invariants and connectivity numbers} (i) with Lemma~\ref{lem: bound for domination number} gives the following homotopical connectivity result on the independence complex of a hypergraph in terms of the order and the maximal degree: $\conn(\IE(\CE))\geq \left \lfloor{\frac{n-1}{2\Delta(\CE)}-1}\right \rfloor$. We will provide an alternative proof as follows:

\begin{proof}[{\bf Proof of Corollary~\ref{cor: homotopic connectivity bound depending vertex number}}] 
  Suppose $\CE$ is a hypergraph with $n_{\CE}$ vertices. Let $f(\CE)\in \Z_{\geq}\cup\{\infty\}$ be the number given by the following:

\begin{equation*}
f(\CE)=
\begin{cases}
\infty, & \text{if $V\neq \emptyset$ and $\CE=\emptyset$},\\
\left \lfloor{\frac{n_{\CE}-1}{2\Delta(\CE)}+1}\right \rfloor, & \text{otherwise}.\\
\end{cases}
\end{equation*}

Assume that $F$ is any edge of the hypergraph $\CE$. Then, since $n_{\CE} = n_{\CE - F}$, $\Delta(\CE) \geq \Delta(\CE -F)$, $\Delta (\CE) \geq \Delta(\CE : F)$ and $|N_{\CE}(F)| \leq |F|(\Delta(\CE)-1)$, we have 

  \begin{align*} f(\CE) = \left \lfloor{\frac{n_{\CE}-1}{2\Delta(\CE)} + 1}\right \rfloor = \left \lfloor{\frac{n_{\CE - F} -1}{2\Delta(\CE)} +1} \right \rfloor \leq \left \lfloor{\frac{n_{\CE - F} -1}{2\Delta(\CE-F)} + 1}\right \rfloor = f(\CE - F)
 \end{align*}
  and

  \begin{align*}  f(\CE) = \left \lfloor{\frac{n_{\CE}-1}{2\Delta(\CE)} + 1}\right \rfloor &= \left \lfloor{\frac{n_{\CE : F} + |F| + |N_{\CE}(F)| -1}{2\Delta(\CE)} + 1}\right \rfloor \\
    & \leq \left \lfloor{\frac{n_{\CE : F} -1}{2\Delta(\CE : F)} +\frac{|F| + |N_{\CE}(F)|}{2\Delta(\CE)} +1}\right \rfloor \\
   & \leq \left \lfloor{\frac{n_{\CE : F} -1}{2\Delta(\CE : F)} +\frac{|F| + |F|(\Delta(\CE)-1)}{2\Delta(\CE)} +1}\right \rfloor \\
   & =\left \lfloor{\frac{n_{\CE : F} -1}{2\Delta(\CE : F)} +\frac{|F|}{2} +1}\right \rfloor \\
   & \leq \left \lfloor{\frac{n_{\CE : F} -1}{2\Delta(\CE : F)} +|F| -1 +1 }\right \rfloor \\
   & =  f(\CE : F) + |F| -1.
  \end{align*}
 The result then follows from Theorem~\ref{thm: connectivity bound for any given hypergraph} and Lemma~\ref{lem: hereditary value of a hypergraph}.

\end{proof}

\begin{rem}
For the particular case where $\CE$ is a graph, Corollary~\ref{cor: homotopic connectivity bound depending vertex number} reduces to the following result of Engstr\"om in~\cite{En}:  If $G$ is a graph with $n$ vertices and maximal degree $d$, then $\IE(G)$ is $\left \lfloor{\frac{n-1}{2d}-1}\right \rfloor$-connected. 
\end{rem}

\subsection{Properly-splitted Hypergraphs}
In this subsection, we introduce the class of properly-splitted hypergraphs which properly contains the class of chordal graphs. Recall that chordal graphs satisfy the Aharoni-Berger-Ziv conjecture (\cite{Kawamura}, Proposition 3.3. and~\cite{A}, Corollary 5.6). We show that, for any properly-splitted hypergraph $\CE$, we have $\psi(\CE) = \conn(\IE(\CE)) + 2$.

The following is motivated by the Proof of Proposition 3.3. in~\cite{Kawamura}.

\begin{lem}\label{lem: relationship between connectivities of CE and CE: F}
Let $\CE$ be a hypergraph on $V$ and let $F$ be an edge of $\CE$.
\begin{enumerate}[label=(\roman*)]
\item If $\connH(\IE(\CE - F)) > \connH(\IE(\CE))$, then $\connH(\IE(\CE : F))\geq \connH(\IE(\CE)) - |F| + 1$.

\item If $\connH(\IE(\CE : F)) \geq \connH(\IE(\CE)) - |F| + 1$, then $\connH(\IE(\CE -F)) \geq \connH(\IE(\CE))$.
\end{enumerate}
\end{lem}

\begin{proof}
The proof is immediate by Theorem~\ref{thm: exact sequence of homology groups-flagization}.
\end{proof}

\begin{defn}
We call a hypergraph $\CE$ a \emph{properly-splitted hypergraph} if, either $\CE= \emptyset$, or $\CE$ admits an edge $F$ satisfying $\connH(\IE(\CE : F))\geq \connH(\IE(\CE)) - |F| + 1$ such that both the hypergraphs $\CE - F$ and $\CE : F$ are properly-splitted hypergraphs.
\end{defn}

\begin{thm}\label{thm: homological connectivity bound for properly-splitted hypergraphs}
If $\CE$ is a properly-splitted hypergraph, then we have $\psi(\CE) = \connH(\IE(\CE)) + 2$.
\end{thm}

\begin{proof}

Suppose that $\CE$ is a properly-splitted hypergraph. If $\CE=\emptyset$, then the claim is obviously true. Assume now that $|\CE|\geq 1$. Then, there exists an edge $F$ of $\CE$ satisfying $\connH(\IE(\CE : F))\geq \connH(\IE(\CE)) - |F| + 1$ such that both the hypergraphs $\CE - F$ and $\CE : F$ are properly-splitted hypergraphs. Let $\psi'$ be a function from the family of hypergraphs to the set $\Z_{\geq}\cup\{\infty\}$ defined by $\psi' (\DE):= \connH(\IE(\DE)) + 2$ for each hypergraph $\DE$. This, together with Lemma~\ref{lem: relationship between connectivities of CE and CE: F} (ii), implies that $\psi'(\CE) \leq \psi'(\CE : F) + |F|-1$ and $\psi'(\CE) \leq \psi'(\CE - F)$. Hence, $\psi(\CE) \leq \connH(\IE(\CE)) +2 = \psi' (\CE) \leq \psi(\CE)$ by Corollary~\ref{cor: homological connectivity bound given by a function} and Lemma~\ref{lem: hereditary value of a hypergraph}. We therefore conclude that $\psi(\CE) = \connH(\IE(\CE)) + 2$. 
\end{proof}

\begin{thm}\label{thm: properly-splitted hypergraphs and homotopy}
If $\CE$ is a properly-splitted hypergraph, then we have $\psi(\CE) = \conn(\IE(\CE)) + 2$.
\end{thm}

\begin{proof}
Suppose $\CE$ is a hypergraph satisfying the property of being properly-splitted. By combining the fact $\connH(\IE(\CE)) \geq \conn(\IE(\CE))$ with Theorem~\ref{thm: connectivity bound for any given hypergraph} and Theorem~\ref{thm: homological connectivity bound for properly-splitted hypergraphs}, we conclude that $\psi(\CE) = \connH(\IE(\CE)) + 2 \geq \conn(\IE(\CE)) + 2 \geq \psi(\CE)$. This completes the proof.
\end{proof}

\begin{rem}
We note that a properly-splitted graph need not be a chordal graph. For example, the cycle graph on $4$ vertices $C_4$ is a properly-splitted hypergraph.
\end{rem}

The following is straightforward to prove by using a simple induction argument.
  
\begin{fact}\label{fact: disjoint union respects properly splitted} The disjoint union of two hypergraphs is properly-splitted if and only if each hypergraph is properly-splitted.
\end{fact}

\section{Homotopy Types of Triangulated Hypergraphs}
In this section, we compute the homotopy types of independence complexes of $d$-uniform properly-connected triangulated hypergraphs and prove that a $d$-uniform properly-connected triangulated hypergraph is necessarily a properly-splitted hypergraph.

The following is implicit in~\cite{Tuyl}.

\begin{thm}\label{thm: some properties of d-uniform properly connected hypergraphs}Let $\CE$ be a $d$-uniform properly-connected hypergraph on $V$ and $F$ being an edge of $\CE$. Then, we have the following:

\begin{enumerate}[label=(\roman*)]
\item An edge $C$ of $\CE$ is an edge of $\CE : F$ if and only if $dist_{\CE}(F, C) \geq d+1$.
\item $\CE : F = \CE_{V \setminus (F\cup N_{\CE}(F))}$.
\item $\CE : F = \CE^{\geq}$.
\item $\ce_{\CE} \geq \ce_{\CE : F} + 1$, where $\ce_{\CE}$ denotes the maximum number of pairwise $(d + 1)$-disjoint edges of $\CE$.
\end{enumerate}
\end{thm}

\begin{proof} Let $\CE$ be a $d$-uniform properly-connected hypergraph on $V$ and let $F$ be an edge of $\CE$. To prove (i), suppose that an edge $C$ of $\CE$ is an edge of $\CE : F$ with $dist_{\CE}(F, C) \leq d$. Then, by Lemma~\ref{lem: d-uniform properly connected hypergraphs and relabeling} there exists a proper irredundant chain $(E_0 = F, \dots, E_{dist_{\CE}(F, C)} = C)$ and some $y_1 \in N_{\CE}(F)$ satisfying $y_1\in E_i$ for all $1\leq i \leq dist_{\CE}(F, C)$. This contradicts with the assumption on $C$, since $y_1 \notin C$.

  Now suppose that $dist_{\CE}(F, C) \geq d+1$ for some edge $C$ of $\CE$. Assume for contrary that $C$ is not an edge of $\CE : F$. Note that $dist_{\CE}(F, C) \geq d + 1$ implies that $\CE \cap F = \emptyset$. Then, there must exist an edge $K$ of $\CE$ such that $F \cap K \neq \emptyset$ and $K \setminus (F \cap K)\subset C$. We therefore have that $$dist_{\CE}(F, C) \leq d- |F \cap K| + d - |K \setminus (F \cap K)|.$$ Thus $dist_{\CE}(F, C) \leq d$ holds, a contradiction.

Claim (ii) is obviously true, since any edge $C$ of $\CE$ satisfying $1\leq dist_{\CE}(F, C) \leq d$ contains some element from the set $N_{\CE}(F)$ by the argument above.

Claim (iii) is immediate from Claim (i) and (ii).

To verify (iv), let $\E'$ be a family with maximum number of pairwise $(d + 1)$-disjoint edges of $C : F$. Then, $\E'$ is also a family of pairwise $(d + 1)$-disjoint edges of $\CE$ by Claim (ii) and $dist_{\CE}(F, C) \geq d + 1$ for any edge $C \in \E'$ by Claim (i). Thus, $\E' \cup \{F\}$ is a family of pairwise $(d + 1)$-disjoint edges of $\CE$. This completes the proof.

\end{proof}

\begin{proof}[{\bf Proof of Theorem~\ref{thm: triangulatd hypergraphs and homotopy type}}] 

  Suppose that $\CE$ is a $d$-uniform properly-connected triangulated hypergraph on $V$. Let $v\in V$ be a vertex in $\CE$ such that the induced subhypergraph $\CE_{N(v)}$ is $d$-complete and $v$ appears at most twice in any proper irredundant chain in $\CE$. Set $$D_{\CE}(v): = \{I\in \IE(\CE)\colon I \cup \{v\} \in \CE\}.$$ Note that $$\IE(\CE) = \IE(\CE_{\{v\}}) * \lk_{\IE(\CE)}(v) \cup \bigcup_{I\in D_{\CE}(v)} \IE(\CE_I) * \lk_{\IE(\CE)}(I)$$ holds. Note also that $\IE(\CE_{\{v\}}) * \lk_{\IE(\CE)}(v)$ and the complexes $\IE(\CE_I) * \lk_{\IE(\CE)}(I)$ are contractible for any $I\in D_{\CE}(v)$. Moreover, for any $I_1, I_2 \in D_{\CE}(v)$, the intersection $(\IE(\CE_{I_1}) *  \lk_{\IE(\CE)}(I_1)) \cap (\IE(\CE_{I_2}) *  \lk_{\IE(\CE)}(I_2))$ is contained by $\IE(\CE_{\{v\}}) * \lk_{\IE(\CE)}(v)$. We thus have that $$\IE(\CE)\simeq \vee_{I\in D_{\CE}(v)} \Sigma (\mathbb{S}^{|I| - 2} *  \lk_{\IE(\CE)}(I))$$ by Theorem~\ref{union of complexes and homotopy}. Since $\CE_{N(v)}$ is a $d$-complete hypergraph, $\lk_{\IE(\CE)}(I) = \IE(\CE: I\cup \{v\})$ holds. Now, set $F = I\cup \{v\}$. Observe then that $F$ is an edge of $\CE$ that satisfies Claim (i). Observe also that the hypergraph $\CE: I\cup \{v\}$ is a triangulated hypergraph by Theorem~\ref{thm: triangulated hypergraphs and a special splitting edge} and Teorem~\ref{thm: some properties of d-uniform properly connected hypergraphs} (iii). Thus, we deduce, by induction, that $\lk_{\IE(\CE)}(I)$ is either contractible or is homotopy equivalent to a wedge of spheres of dimension at most $(d - 1) \ce_{\CE : I\cup\{v\}} -1$. It then follows that $\IE(\CE)$ is either contractible or is homotopy equivalent to a wedge of spheres of dimension at most $(d - 1) \ce_{\CE : I\cup\{v\}} -1 + |I|$. Theorem~\ref{thm: some properties of d-uniform properly connected hypergraphs} (iv) and the fact $|I| = d -1$ give that $\IE(\CE)$ is either contractible or is homotopy equivalent to a wedge of spheres of dimension at most $(d - 1)\ce_{\CE} - 1$. This proves Claim (ii).
\end{proof}

\begin{lem}\label{lem:a triangulated hypergraph is a properly-splitted hypergraph} Any $d$-uniform properly-connected triangulated hypergraph is a properly-splitted hypergraph.
\end{lem}

\begin{proof} We prove the statement by induction on the number of edges of the hypergraph. Suppose that $\CE$ is a $d$-uniform properly-connected triangulated hypergraph on $V$. Then, there exists a vertex $v\in V$ such that $\CE_{N(v)}$ is a $d$-complete hypergraph and $v$ appears at most twice in any proper irredundant chain in $\CE$. Let $F$ be an edge of $\CE$ containing the vertex $v$. Note that the subhypergraphs $\CE - F$ and $\CE^{\geq} = \{G\in \CE\colon dist_{\CE}(F, G) \geq d+1\}$ satisfy the property of being triangulated by Theorem~\ref{thm: triangulated hypergraphs and a special splitting edge}. It follows by induction that $\CE - F$ is a properly-splitted hypergraph. By Theorem~\ref{thm: some properties of d-uniform properly connected hypergraphs} (iii) and Lemma~\ref{lem: some relations for the number of minimal non-faces} (ii), the induction gives that $C : F$ is a properly-splitted hypergraph. Moreover, we have $\connH(\IE(\CE : F))\geq \connH(\IE(\CE)) - |F| + 1$ (see the proof of Theorem~\ref{thm: triangulatd hypergraphs and homotopy type} (i)). This completes the proof.
\end{proof}

Lemma~\ref{lem:a triangulated hypergraph is a properly-splitted hypergraph} and Theorem~\ref{thm: properly-splitted hypergraphs and homotopy} immediately give the following:

\begin{cor} For a triangulated hypergraph $\CE$, we have $\conn(\IE(\CE)) = \psi(\CE) -2$.
\end{cor}

\section*{Acknowledgement}
I would like to thank Yusuf Civan for his invaluable comments.


\end{document}